\documentclass[fleqn]{amsart}
\usepackage[utf8]{inputenc}

\input{000_setup}

\title[Poncelet: loci and envelopes]{Poncelet triangles: conic loci of the orthocenter and of\\the isogonal conjugate of a fixed point}
\author[R. Garcia]{Ronaldo A. Garcia}
\author[M. Helman]{Mark Helman}
\author[D. Reznik]{Dan Reznik} 

\begin{document}

\begin{abstract}
We prove that over a Poncelet triangle family interscribed between two nested ellipses $\E,\E_c$, (i) the locus of the orthocenter is not only a conic, but it is axis-aligned and homothetic to a $90^o$-rotated copy of $\E$, and (ii) the locus of the isogonal conjugate of a fixed point $P$ is also a conic (the expected degree was four); a parabola (resp. line) if $P$ is on the (degree-four) envelope of the circumcircle (resp. on $\E$). We also show that the envelope of both the circumcircle and radical axis of incircle and circumcircle contain a conic component if and only if $\E_c$ is a circle. The former case is the union of two circles!
\end{abstract}

\maketitle

\vspace{-.3in}
\section{Introduction}
\label{sec:intro}
Consider a Poncelet porism of triangles \cite{dragovic11} inscribed in a first ellipse $\E$ and circumscribing a second, nested one called $\E_c$, \cref{fig:poncelet} (left). We have been interested in loci of triangle centers over Poncelet since we first encountered \cite{zaslavsky2001-poncelet,odehnal2011-poristic}, and are motivated by the continued exploration of related kinematic phenomena, e.g., \cite{dragovic2025-parable,jurkin2024-poncelet,kodrnja2023-locus,koncul2018-isot}.

Recently, we have uncovered many locus harmonies for the case where $\E_c$ is a circle, let $C$ denote its center \cite{garcia2024-incircle}. For example, the locus of the circumcenter $X_3$ (we will be using a notation after Kimberling \cite{etc}) is an ellipse with each focus on a major axis of $E$, the latter passing through $C$. In said article, we rely on the following results for the case of generic $\E,\E_c$ which we will be proving here, namely:

\begin{figure}
\centering
\includegraphics[width=\linewidth]{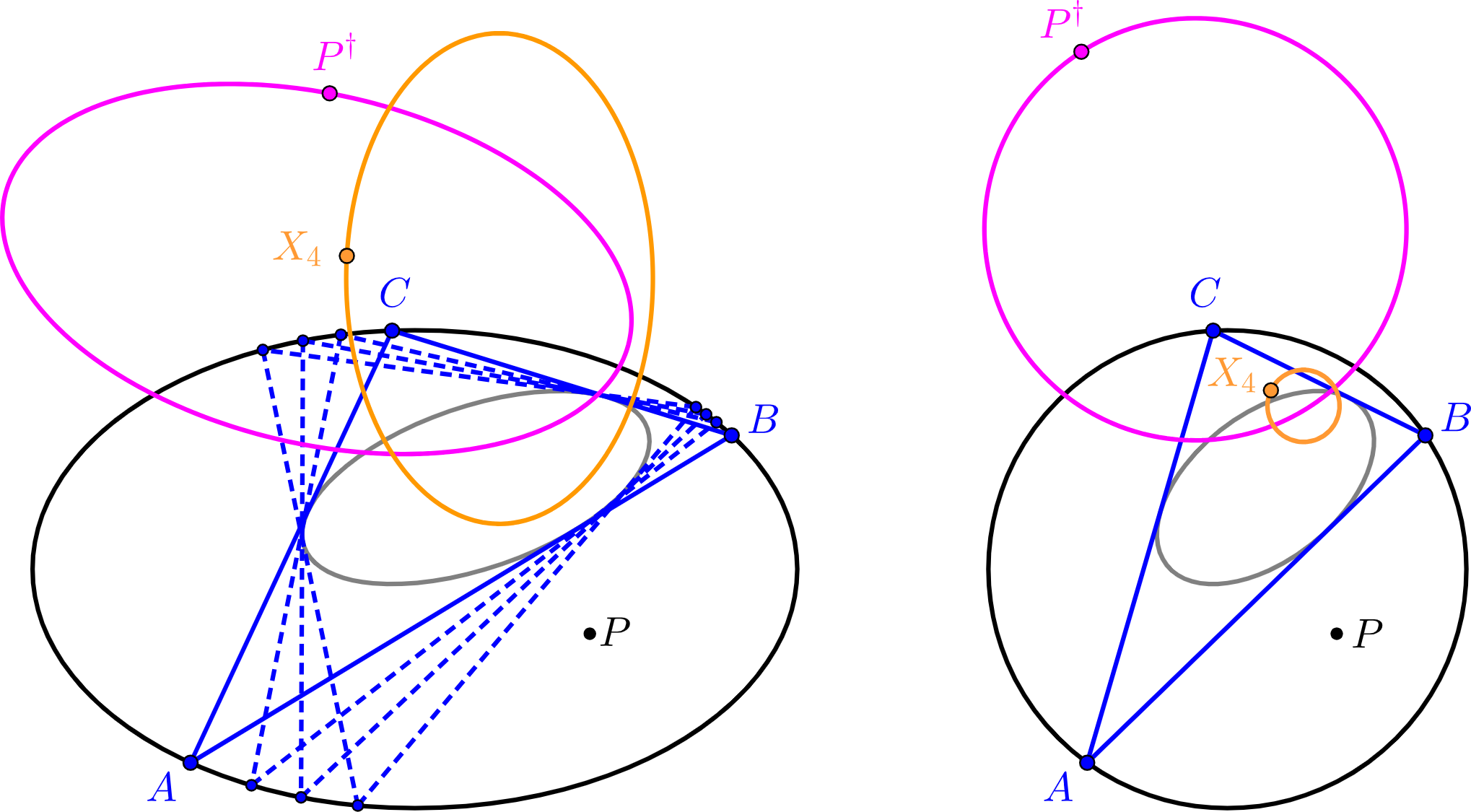}
\zcaption{\tb{left}: Poncelet triangles ABC inscribed in an outer conic and circumscribing an inner, nested one. Also shown are the loci over the family of (i) the orthocenter $X_4$ (orange), and (ii) the isogonal conjugate $gP$ (purple) of a fixed point $P$. \tb{right}: a similar setup where the outer conic is a circle. In this case both (i) and (ii) are circles, the latter result proved in \cite{skutin2013-isogonal}. Video: \hrefs{https://youtu.be/sv59VPzdUCs}}
\label{fig:poncelet}
\end{figure}

\begin{itemize} 
\item \cref{sec:locus-ortho}: the locus of the orthocenter -- shown in \cite{helman2021-power-loci,helman2021-theory} to be always a conic -- is a peculiar ellipse: it is axis-parallel and homothetic to a $90^o$-rotated copy of $\E$.
\item \cref{sec:locus-isog}: referring to \cref{fig:poncelet} (right), in \cite{skutin2013-isogonal} it is shown that if $\E$ is a circle, the \ti{isogonal conjugate}\footnote{A type of inversive transformation first studied by J. Neuberg and E. Lemoine in the 1880s.}, -- see \cref{fig:isog-constr} (left) and \cite{akopyan2012-conjugation,dean2001-conjug,sigur2005-conjug}-- of a fixed point with respect to Poncelet triangles sweeps a circle, \cref{fig:poncelet} (right). Here we let $\E$ be a generic ellipse, and find that the same locus -- note that isogonal conjugation is not affinely invariant -- is a conic. Interestingly, said locus degeneratesto a parabola if it lies on the degree-4 boundary of the region swept by the circumcircle, \cref{sec:region}, and to a line if $P$ lies on $\E$. Furthermore, over all $P$ on $\E$, said line-loci envelop an ellipse, concentric with $\E_c$.
\item \cref{sec:circ-env}: as an extension to \cite{garcia2024-incircle} -- Poncelet triangles about a circular $\E_c$ (their incircle) -- we describe the remarkable envelopes of both the circumcircle and of the radical axis of incircle and circumcircle. The former is the union of two circles and the latter is a conic. We conjecture that either envelope has a conic component if and only if $\E_c$ is a circle.
\end{itemize}

Our proofs are based on analytic geometry, often producing long, explicit expressions for the loci or envelopes with the aid of a Computer Algebra System (CAS). The reader is encouraged to replace them for a more concise argument based on the tools of algebraic geometry described in \cite{zaslavsky2001-poncelet} and deployed in \cite{olga2014-incenters,schwartz2016-com}.

\section{Symmetric parametrization}
\label{sec:symmetric}
Identifying $\mathbb{R}^2$ with $\mathbb{C}$, consider the following parameterization for Poncelet triangles inscribed in $\mathbb{T}$, the unit circle centered at the origin, as derived in \cite[Def. 3]{helman2021-power-loci} and based on the work in \cite{daepp2019-ellipses} on Blaschke products:
\begin{theorem}
For any Poncelet family of triangles inscribed in the unit circle $\mathbb{T}$ and circumscribing a nested ellipse with foci $f,g\in\mathbb{D}$ (the unit disk), parametrize its vertices $z_1,z_2,z_3\in\mathbb{T}$ as the following elementary symmetric polynomials:
\begin{align*}
z_1+z_2+z_3=& f+g+\l\ol f \ol g , \\
z_1 z_2+z_2 z_3+z_3 z_1=& f g+\l(\ol f+\ol g),\\
z_1 z_2 z_3=& \l,
\end{align*}
where the free parameter $\l=e^{i \theta}$, $\theta\in[0,2\pi]$.
\label{SymPar}
\end{theorem}

This is generalized to a Poncelet triangle family $\T$ interscribed between any two nested ellipses $\E,\E_c$ by applying an affine transformation that sends $\mathbb{T}$ to $\E$. Let $z_1,z_2,z_3\in\E$ be the varying vertices of the Poncelet triangles. The statements below are reproduced from \cite[Sec.2.2]{garcia2024-incircle}, where their proofs can also be found: 

\begin{theorem}
For any symmetric rational function $\F:\mathbb{C}^3\rightarrow\mathbb{C}$, the value of $\F(z_1,z_2,z_3)$ can be parameterized as a rational function of a parameter $\lambda$ on $\mathbb{T}$.
\end{theorem}

Let $a,b$ denote the semiaxis' lengths of $\E$, i.e., it satisfies $(x/a)^2+(y/b)^2=1$. Consider the affine transformation $\A(x,y)=(a x,b y)$ which sends the unit circle into $\E$. So $\A^{-1}(x,y)=(x/a,y/b)$. In the complex plane, $\A(z):=\frac{(a+b)}{2}z+\frac{(a-b)}{2}\overline{z}$. $\A^{-1}(z)=\frac{(1/a+1/b)}{2}z+\frac{(1/a-1/b)}{2}\overline{z}$. Let $c^2=a^2-b^2$. 

\begin{lemma}
$\E_{pre}:=\A^{-1}(\E_c)$ is an axis-aligned ellipse with semi-major axis $r/b$ and semi-minor axis $r/a$, center $\A^{-1}(C)=x_c/a+i y_c/b$, and semi-focal length $r\frac{c}{a b}$, with foci given by $x_c/a+i (y_c/b\pm r\frac{c}{a b})$.
\label{CircleTransformLemma}
\end{lemma}

\begin{corollary}
When the inner ellipse $\E_c$ of the Poncelet triangle family is a circle, the sum and product of the foci of $\E_{pre}$ as complex numbers are given by
\begin{align*}
f_{pre}+g_{pre}=&\frac{2x_c}{a}+\frac{2y_c}{b}\rc\\
f_{pre} g_{pre}=& \frac{a^2+b^2}{c^2}+\frac{2i}{a b}\left(x_c y_c+\frac1c \sqrt{(a^4-c^2 x_c^2)(b^4-c^2 y_c^2)}\right)\cdot
\end{align*}
\end{corollary}

\section{Orthocenter locus}
\label{sec:locus-ortho}
Let $\T$ be as defined in \cref{sec:symmetric}. Let $O_c=[x_c,y_c]$ be the center of $\E_c$. Without loss of generality, let the semiaaxes of $\E$ be along $x$ and $y$ directions and $a,b$ denote their lengths.

Referring to \cref{fig:affine}, let $\mathcal{A}$ be the affine transformation that sends $\E$ to the unit circle in $\mathbb{C}$, i.e., $(x,y)\rightarrow (x/a+i y/b)$. Let $f=f_x+i f_y, g=g_x+i g_y$ be the foci of the image of $\E_c$ under $\mathcal{A}$. Referring to \cref{fig:gen-x4}:

\begin{figure}
\centering
\includegraphics[width=.9\linewidth]{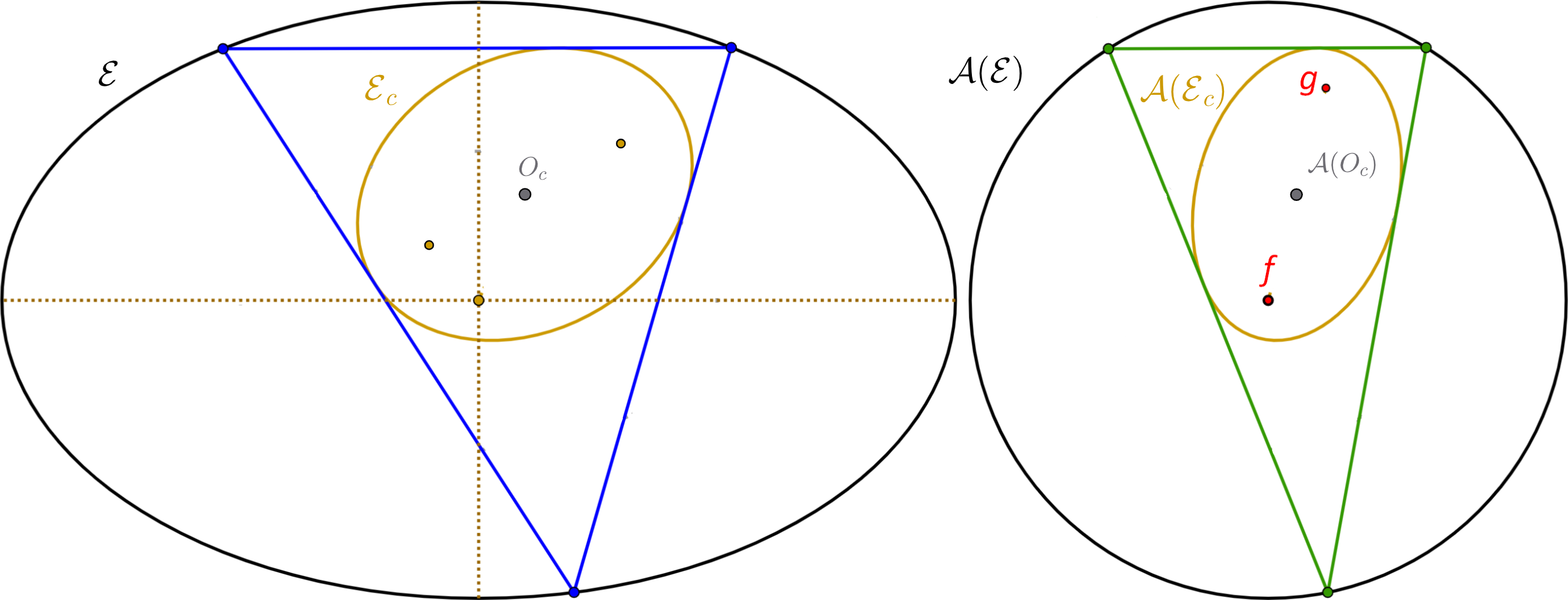}
\zcaption{\tb{left}: a Poncelet triangle (blue) inscribed in $\E$ and circumscribing $\E_c$. Let $\O_c$ denote its center. \tb{right}: image of the former under an affine transformation $\mathcal{A}$ that sends $\E$ to the unit circle $\mathcal{A}(\E)$ in the complex plane. Let $f,g$ be the foci of the caustic image $\mathcal{A}(\E_c)$. Notice that its center is simply $\mathcal{A}(\O_c)$ since conic centers are equivariant under affine transformations.}
\label{fig:affine}
\end{figure}

\begin{theorem}
\label{thm:locus-ortho}
Over $\T$, the locus of $X_4$ is an ellipse homothetic to a $90^o$-rotated copy of $\E$. The locus center $C_4$ only depends on $a,b$ of $\E$ and $O_c=[x_c,y_c]$ and is given by:
\begin{align*} 
C_4&=(a^2 + b^2)\left[\frac{x_c}{a^2}, ~\frac{y_c}{b^2}\right]\cdot
\end{align*}
The semiaxes $a_4,b_4$, $a_4{\geq}b_4$ (the former parallel to $\E$'s minor axis) do depend on the relative position of $\E$ and $\E_c$, and are given by:
\[ a_4 =\frac{\sigma}{2b}\rc\quad b_4=\frac{\sigma}{2a}\rc\]
where $\sigma=\left| f g \left(a^{2}+b^{2}\right)-c^{2}\right|{\geq}0$. Note that as claimed, $a_4/b_4=a/b$.
\end{theorem}



\begin{figure}
\centering
\includegraphics[width=.9\linewidth]{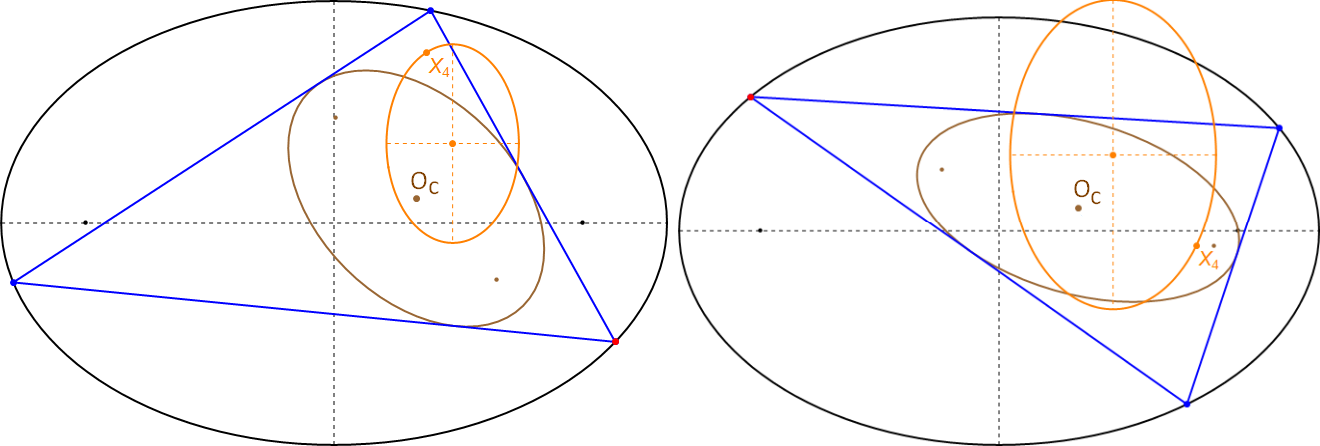}
\zcaption{\tb{left}: a Poncelet triangle (blue) interscribed between an outer ellipse $\E$ and a generic caustic $\E_c $, let $O_c$ be its center. The locus of $X_4$ (orange) is homothetic to a $90^o$-rotated copy of $\E$. \tb{right}: another caustic centered at the same $O_c$ which closes Poncelet. Notice that the center of the (new) locus remains at the same location. Video: \hrefs{https://youtu.be/tGZa3p6Q1BA}}
\label{fig:gen-x4}
\end{figure}

\begin{proof}
Obtain the locus explicitly by plugging the barycentric coordinates of $X_4$, namely, $1/(l_2^2 + l_3^2 - l_1^2)::cyc.$, $l_i$ are the sidelengths of a triangle \cite{etc}, into the symmetric/rational parametrization for Poncelet triangles described in \cref{sec:symmetric}. Upon CAS simplification, obtain  a conic with center at:
\[(a^2 + b^2) \left[\frac{(f_x+g_x)}{2a} , \frac{ (f_y+g_y)}{2b}\right]\rc \]
\noindent which is identical to what's claimed since conic centers are equivariant under affine transformations and  $(f_x+g_x)/(2a)=x_c$ and $(f_y+g_y)/(2b)=y_c$.

Similarly, obtain the semiaxes' numerator $\sigma$ as: 
\begin{align*}
\sigma^2=& |fg|^2\left(a^{2}+b^{2}\right)^{2}- c^{2} \left(a^{2}+b^{2}\right) \left(f g +\overline{fg} \right) +c^{4} \\
=&\left(f g \left(a^{2}+b^{2}\right)-c^{2}\right) \left(\overline{fg}\left(a^{2}+b^{2}\right)-c^{2}\right)>0,
\end{align*}
\noindent which simplifies to the claim.
\end{proof}

\section{Isogonal locus}
\label{sec:locus-isog}
\begin{definition}[Isogonal conjugate]
The isogonal conjugate $\gP$ of a point $P$ with respect to a triangle $T$ is the point of concurrence of the \ti{cevians} of $P$ (lines from each vertex through $P$) reflected upon the angle bisectors. 
\end{definition}

In particular, isogonal conjugation sends a triangle's circumcircle to the line at infinity (and vice-versa) \cite[Isogonal Conjugate]{mw}. If the barycentrics of $P$ are $[z_1:z_2:z_3]$, $\gP=[l_1^2/z_1 : l_2^2/z_2 : l_3^2/z_3]$, where the $l_i$ are the sidelengths. An alternative construction is shown in \cref{fig:isog-constr}, based on the shared circumcircle of the pedal triangles of $P,\gP$. For more information, see \cite{dean2001-conjug,sigur2005-conjug}.

\twofigs{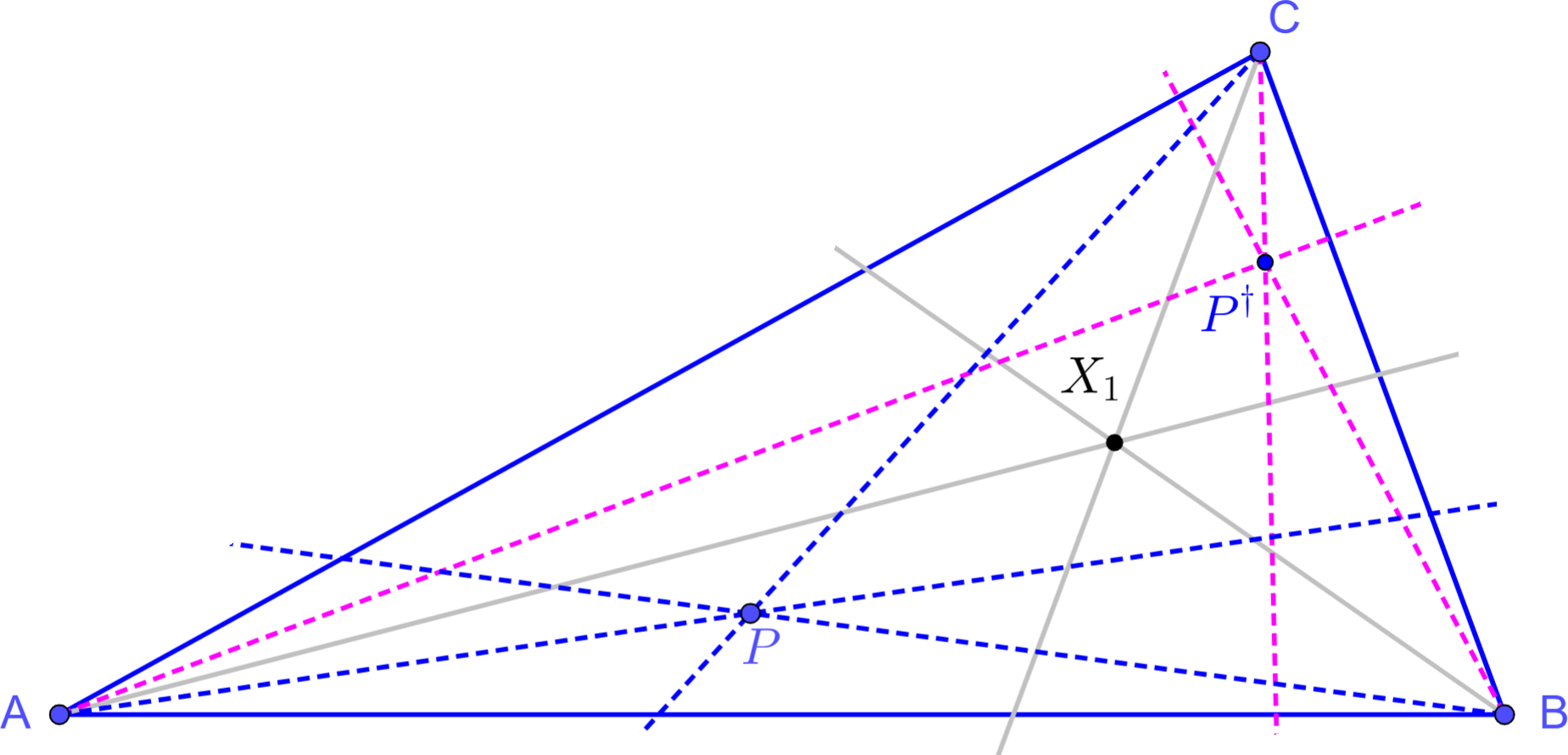}{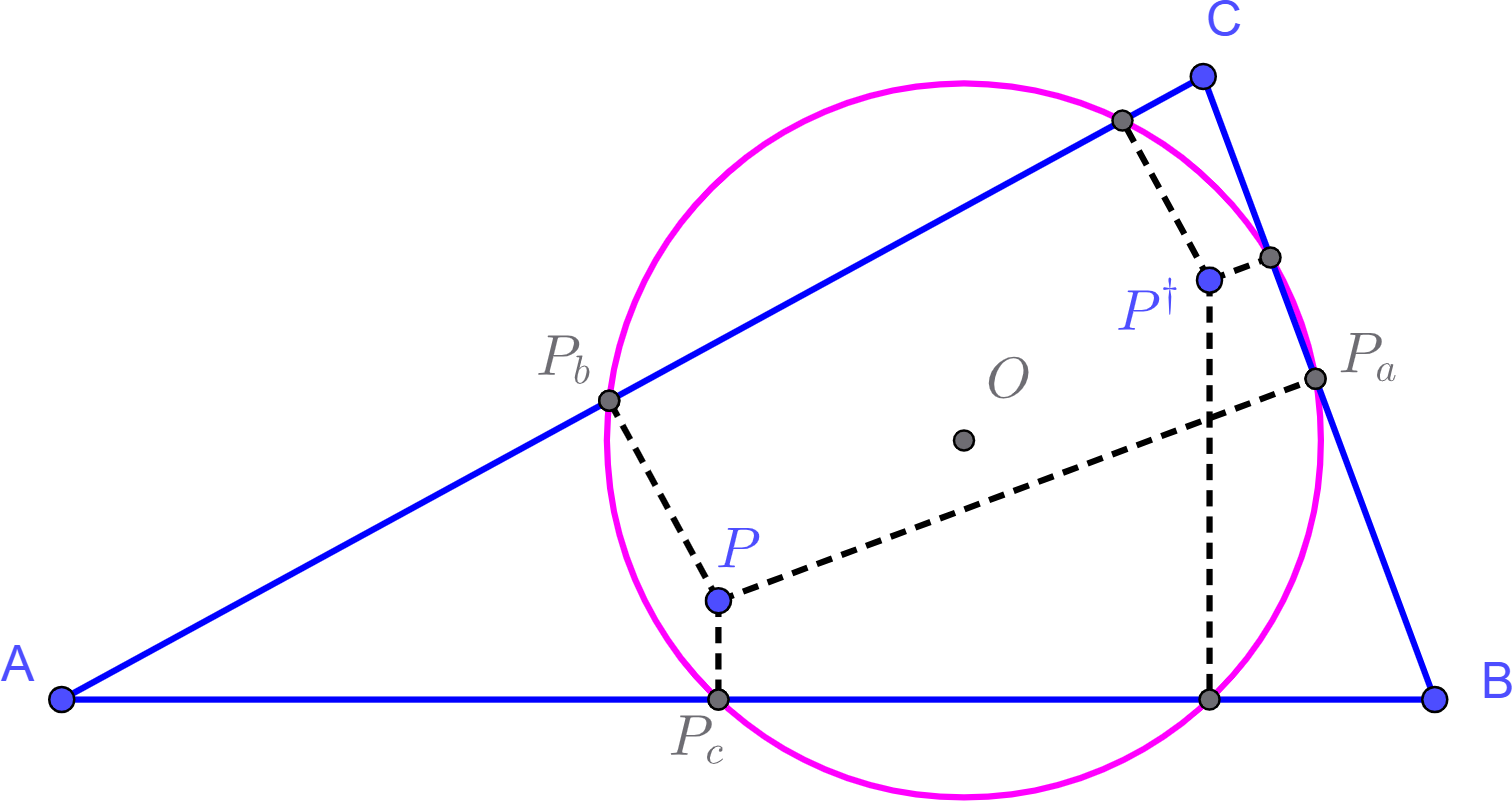}{\tb{left:} The isogonal conjugate $\gP$ of $P$ with respect to $T=ABC$ can be obtained by (i) drawing the three cevians through $P$ (dashed blue), (ii) reflecting each about the corresponding angular bisector (gray, meet at the incenter at $X_1$), and (iii) locating their intersection. \tb{right:} an alternative method, used here, consists in (i) obtaining the three perpendicular projections $P_a,P_b,P_c$ of $P$ on the sides of $T$; (ii) computing the center $O$ of the circle (purple) passing through said projections; (iii) locating $\gP$ at the reflection of $P$ about $O$. Notice that the perpendicular projections from $\gP$ onto $T$ lie on the same circle.}{fig:isog-constr}


Let $P$ be a fixed point. Let $\gL$ denote the locus of $\gP$ with respect to triangles in $\T$.

Consider the special case of $\E$ the unit circle, and $\E_c$ with foci $f,g$ interior to $\E$. In \cite{skutin2013-isogonal} it is proved that $\gL$ is a circle. We extend this computing its center and radius explicitly.

\begin{proposition}
$\gL$ is a circle centered on $O^{\dagger}$ and of radius $r^{\dagger}$ given by:
\begin{equation}
O^{\dagger} = \frac{(f+g-f g\ol{P} - P)}{1-|P|^2},\quad r^{\dagger}= \left|\frac{(\ol{g} - \ol{P})(\ol{f} - \ol{P})}{1-|P|^2}\right|\rd
\end{equation}
\label{prop:skutin-weaver}
\end{proposition}

\begin{proof}
Let $\{z_1,z_2,z_3\}$ be a triangle with $|z_i|=1$. The following expression is given in \cite{Weaver1935} for the isogonal conjugate of $P\in\mathbb{C}$:
\begin{equation} \gP=\frac{\ol{P}^2\sigma_3-\ol{P}\sigma_2+\sigma_1-P}{1-|P|^2}\rc
\label{eqn:isog-weaver}
\end{equation}
where $\sigma_1=z_1+z_2+z_3$, $\sigma_2=z_1 z_2+z_2 z_3 + z_3 z_1$, and $\sigma_3=z_1 z_2 z_3$, i.e., identical to the left hand sides in \cref{SymPar}. 

Expanding \cref{eqn:isog-weaver} with \cref{SymPar}, obtain its center $O^{\dagger}$ and radius $|r^{\dagger}|$ explicitly:
\[\gP= \frac{\left(f+g-f g\ol{P} - P\right)}{1-|P|^2} + \frac{(\ol{g} - \ol{P})(\ol{f} - \ol{P})\lambda}{1-|P|^2}\rd\qedhere\]
\end{proof}

Nevertheless, the isogonal conjugate is not equivariant under an affine transform, therefore we need another approach for the case when $\E$ is not a circle.

Let $x,y,z$ be the complex vertices of a triangle and $P$ be a point.

\begin{lemma}
$\gP=\alpha/\beta$, with:
\begin{align*}
\alpha=&\left[(\ol{x} - \ol{y}) x y + (-\ol{x} + \ol{y}) x y + (\ol{y} - \ol{y}) y y\right] \ol{P} P \\
-& \left[ \ol{y} (\ol{x} - \ol{y}) x y + \ol{y} (\ol{y} - \ol{x}) x y + \ol{x} (\ol{y} - \ol{y}) y y\right] P \\
-& (y - y) (x - y) (-y + x) \ol{P}^2 \\
+& \left[(\ol{y} + \ol{y}) x^2 y + (-\ol{y} - \ol{y}) x^2 y + (-\ol{x} - \ol{y}) x y^2 \right.\\
+& \left.(\ol{x} + \ol{y}) x y^2 + (\ol{x} + \ol{y}) y^2 y + (-\ol{x} - \ol{y}) y y^2\right] \ol{P} \\
-& \ol{y}\ol{y} (y - y) x^2 + \ol{x} (y^2 \ol{y} - \ol{y} y^2) x - \ol{x} y y (y \ol{y} - \ol{y} y),\\
\beta=& \left[(-\ol{y} + \ol{y}) x + (\ol{x} - \ol{y}) y + (-\ol{x} + \ol{y}) y\right] \ol{P} P \\
+& \left[\ol{x} (\ol{y} - \ol{y}) x - \ol{y} (\ol{x} - \ol{y}) y + \ol{y} (\ol{x} - \ol{y}) y\right] P\\
+& \left[(-\ol{x} + \ol{y}) x y + (\ol{x} - \ol{y}) x y + (-\ol{y} + \ol{y}) y y\right] \ol{P} \\
+& \ol{y} (\ol{x} - \ol{y}) x y - \ol{y} (\ol{x} - \ol{y}) x y + \ol{x} (\ol{y} - \ol{y}) y y \ldotp
\end{align*}
\label{lem:isog-complex}
\end{lemma}

\begin{proof}
The complex foot $P_{\perp}$ of the perpendicular from $P$ onto a line $z z'$ is given by:
\[ P_{\perp} = \frac{(P - z)\,\ol{(z' - z)}}{|z' - z|^2}(z' - z) + z. \]
\noindent An expression for the complex circumcenter $O$ of a triangle with vertices $\{x,y,z\}$ in terms of a ratio of determinants is given in \cite[Lemma 6.24, p.108]{chen2024-euclidean}:
\[
O =
\left| \begin{array}{ccc}
x & x \ol{x} & 1 \\
y & y \ol{y} & 1 \\
z & z \ol{z} & 1 \\
\end{array} \right|
\div
\left| \begin{array}{ccc}
x & \ol{x} & 1 \\
y & \ol{y} & 1 \\
z & \ol{z} & 1 \\
\end{array} \right|.
\]
The construction in \cref{fig:isog-constr} implies $\gP=2O-P$, where $O$ is the circumcenter of the pedal triangle. Use the above to obtain it explicitly. The claim results from simplification with a CAS.
\end{proof}

\begin{theorem}
\label{thm:locus-isog}
$\gL$ is a conic.
\end{theorem}

\begin{proof} 
Consider the affine transformation
\[\A(z)=\frac{(a+b)}{2}z +\frac{(a-b)}{2}\overline{z}=(ax,by),\quad z=x+i y,\]
and a triangle $z_1,z_2,z_3$ inscribed in the unit circle.

The vertices $A,B,C$ of a Poncelet triangle are given by $\A(z_i)$, $i=1,2,3$. Applying \cref{SymPar} to \cref{lem:isog-complex}, obtain:
\begin{align}
\gP=\frac{s_2\lambda^2+s_1\lambda+s_0 }{t_2\lambda^2+t_1\lambda+t_0},\quad \lambda\in \mathbb{T}.
\label{eqn:isog-ratio}
\end{align}

Expressions for $s_i$ and $t_i$ in terms of $a,b,f,g,P$ can be obtained via CAS and appear in \cref{app:factors}. After eliminating of  $\lambda=e^{i\theta}$, we obtain an algebraic curve of degree 4 of the form:
\begin{equation}
k_{40} x^4 + k_{22} x^2 y^2 + k_{04} y^4 + k_{20} x^2 + k_{11} x y + k_{02} y^2 + k_{10} x + k_{01} y + k_{00}=0,
\label{eqn:conic-eqn}
\end{equation}

We omit the rather long symbolic expressions for the $k_{ij}$. Simplification with a CAS yields that $k_{40}=k_{22}= k_{04}=0$, i.e., the curve is of at most second degree. 
\end{proof}

\subsection{Region swept by circumcircle}
\label{sec:region}
Let $\Rm$ denote the region swept by the circumcircle over the family. Referring to \cref{fig:envelope-algebraic}:

\begin{figure}
\centering
\includegraphics[width=0.7\linewidth]{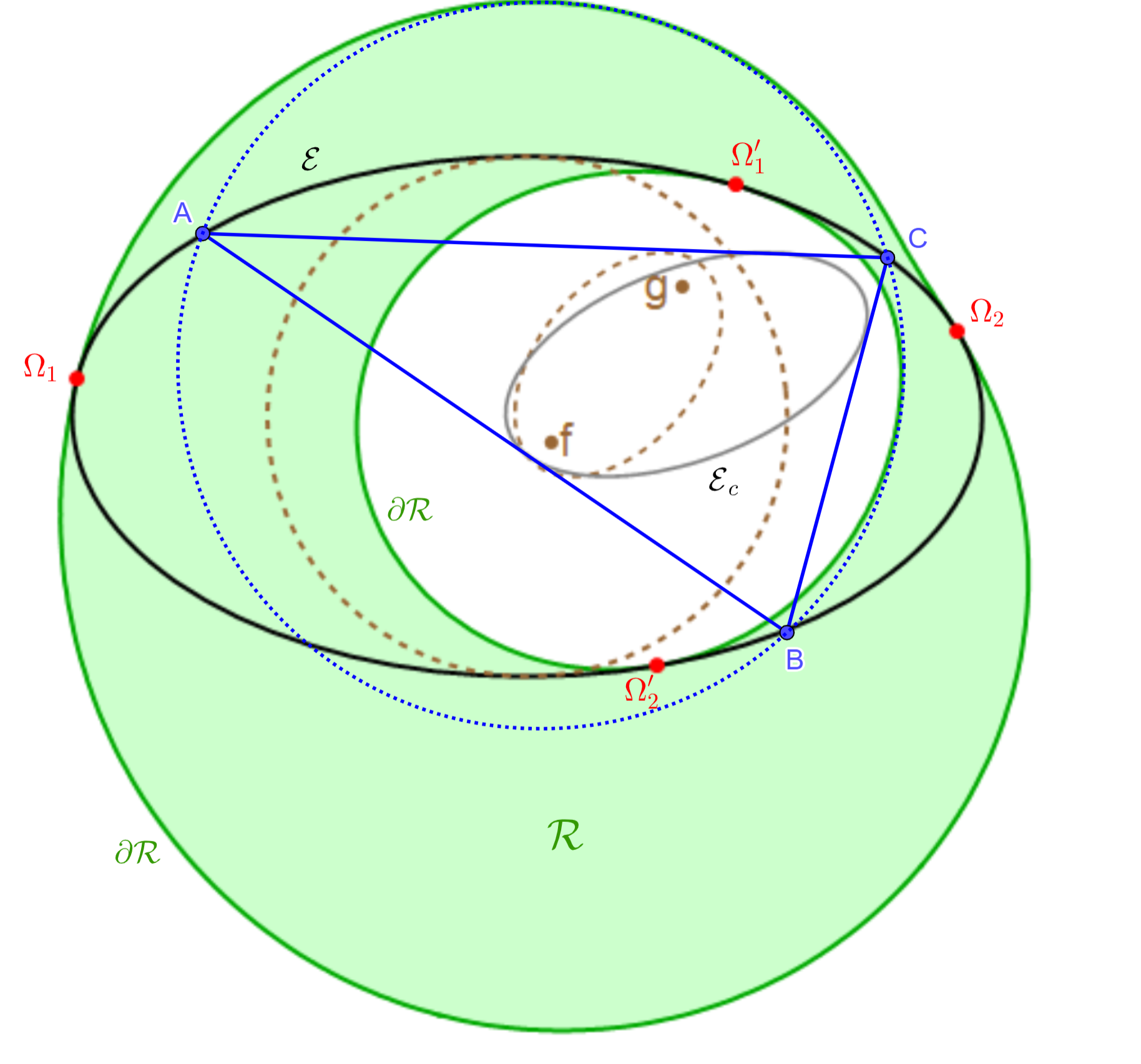}
\zcaption{The region $\Rm$ swept by the circumcircle is shown in green. It is bounded by the envelope of the circumcircle $\partial\Rm$, given by the union of two degree-four curves. These are externally (resp. internally) tangent to $\E$ ($a=1.75$, $b=1$) at two points $\Omega_1,\Omega_2$ (resp. $\Omega_1',\Omega_2'$). Shown superposed (dashed brown) is the system inscribed in $\mathbb{T}$ and its caustic with foci $f,g$.}
\label{fig:envelope-algebraic}
\end{figure}

\begin{proposition}
\label{prop:env}
The boundary $\partial\Rm$ is the union of two closed curves given implicitly by a polynomial of degree 4. One (resp. the other) is externally (resp. internally) tangent to $\E$ at two points.
\end{proposition}

\begin{proof}
When $P$ is exterior (resp. interior) to $\Rm$, the circumcircle will never (resp. sometimes) pass through it, i.e., the conic will be an ellipse (resp. hyperbola). Therefore $\partial\Rm$ corresponds to zeroes of the determinant of the Hessian $\mathcal{H}$ extracted from \cref{eqn:conic-eqn}. Namely, setting $\det(\mathcal{H}) = 4k_{20}k_{02} - k_{11}^2 = 0$ yields the implicit equation shown in \cref{app:envelope-circum}.

To obtain the four points of tangency, define 
$\E$ via a rational parametrization:
\begin{equation}
\E(t) = \left[a\cdot \frac{1 - t^2}{1 + t^2}, \quad
b\cdot\frac{2t}{1 + t^2}\right].
\label{eqn:rational}
\end{equation}

Now evaluate $\partial\Rm=0$ given in \cref{app:envelope-circum} at $(x,y)=\E(t)$. After simplification, obtain:
\begin{align*}
&\left[\left( f_x g_y + f_y g_x + f_y + g_y \right)t^4 
+ \left( -2 f_x - 2 g_x - 4 \right)t^3+ \left( 2 f_x g_y + 2 f_y g_x \right)t^2\right. \\
&\left.+ \left( -2 f_x - 2 g_x + 4 \right)t + g_y f_x + f_y g_x - f_y - g_y\right]^2=0.
\end{align*}
Since the left-hand side is a square, there are four quadratic tangents, given by the four roots of the radicand. It can be shown that with $f,g$ interior to $\mathbb{T}$ these are always real. 
\end{proof}

\twofigs{pics/pics_350_p_ellipse_in}{pics/pics_360_p_ellipse_out}{$P$ exterior to $\Rm$ (green region), the locus is an ellipse crossing $\E$ at two points $Z_1,Z_2$. \tb{left}: $P$ is in the internal connected component of the exterior of $\Rm$; $Z_1$ (resp. $Z_2$) is the apex of the Poncelet triangle with base $B_1 C_1$ passing through $P$ and along a first (resp. second) tangent to $\E_c$. \tb{right}: $P$ is on the exterior connected component of the exterior of $\Rm$. A similar construction for $Z_1,Z_2$ applies, based on the two tangents from $P$ to $\E_c$. Video: \hrefs{https://youtu.be/v\_K0xoQy4IM}}{fig:exterior-R}

\subsection{Conic type}

The specific conic type of $\gL$ (ellipse, hyperbola, etc.) correspond to the following corollaries based on \cref{thm:locus-isog} and \cref{prop:env-circumcircle}. Referring to \cref{fig:exterior-R}, \cref{fig:on-slight} (right), and \cref{fig:two-branches}:
\begin{corollary}
If $P$ is exterior (resp. interior) to $\Rm$, $\gL$ is an ellipse (resp. hyperbola) which intersects $\E$ at two points $Z_1,Z_2$. These are the apexes of Poncelet triangles whose bases $B_1 C_1$ and $B_2 C_2$ are along the two tangents through $P$ onto $\E$.
\label{cor:exterior-R}
\end{corollary}

\begin{proof}
When $P$ is exterior (resp. interior) to $\Rm$, it is never touched (sometimes touched) by the moving circumcircle, yielding a conic $\gL$ comprised of finite points only (resp. mostly finite, with some infinite points).
\end{proof}

Referring to \cref{fig:on-slight} (left):
\begin{corollary}
If $P$ is on $\partial\Rm$, $\gL$ is a parabola.
\end{corollary}

\begin{proof}
This stems from continuity between the two states of $P$ (interior, exterior to $\Rm$) in \cref{cor:exterior-R}.
\end{proof}

\twofigs{pics/pics_330_p_on_caustic_1}{pics/pics_340_p_on_caustic_2}{$P$ is on $\E_c$, locus is tangent to $\E$ at one point $Z$. \tb{left}: a first position for $P$. The point $Z$ is the apex of a Poncelet triangle with base $B_1 C_1$, tangent to $\E_c$ at $P$. \tb{right}: a second example. Video: \hrefs{https://youtu.be/v\_K0xoQy4IM}}{fig:on-caustic}

Referring to \cref{fig:on-caustic}:

\begin{corollary}
If $P$ is on $\E_c$, the $\gL$ ellipse touches $\E$ at a single point $Z$ which is the apex of a Poncelet triangle whose base is tangent to $\E_c$ at $P$.
\label{cor:on-caustic}
\end{corollary}

\begin{proof}
The isogonal conjugate of any point on a side of a triangle is the opposite vertex.
\end{proof}

\twofigs{pics/pics_310_p_in_caustic}{pics/pics_320_p_on_focus}{$P$ interior to $\E_c$, elliptic locus of $\gP$ is doesn't intersect $\E$. The green region is $\Rm$, i.e., what is swept by the circumcircle of $ABC$ over Poncelet. \tb{left}: $P$ at a generic position within $\E_c$ \tb{right}: when that point is a focus of $\E_c$, $\gP$ is stationary on the other focus since the foci of an inconic are a conjugate pair \cite{beluhov2007-isogonal-foci}.}{fig:in-caustic}

Referring to \cref{fig:in-caustic}:
\begin{corollary}
If $P$ is interior to $\E_c$, the $\gL$ ellipse is interior to and does not intersect $\E$. Furthermore, if $P$ is at a focus of $\E_c$, $\gL$ becomes stationary at the other focus.
\end{corollary}

\begin{proof}
This stems from continuity from \cref{cor:exterior-R} and \cref{cor:on-caustic}. The second part stems from the fact that the foci of an inconic are isogonal conjugates \cite{beluhov2007-isogonal-foci}.
\end{proof}

\twofigs{pics/pics_370_p_parabola}{pics/pics_380_p_hyp_near}{\tb{left}: if $P$ is on the boundary of $\Rm$, $\gL$ is a parabola which intersects $\E$ at two points $Z_1,Z_2$ constructed as before; \tb{right}: by nudging $P$ slightly into $\Rm$ (green region), the parabola becomes a hyperbola (only one branch is shown).}{fig:on-slight}

\begin{figure}
\centering \includegraphics[width=0.7\linewidth]{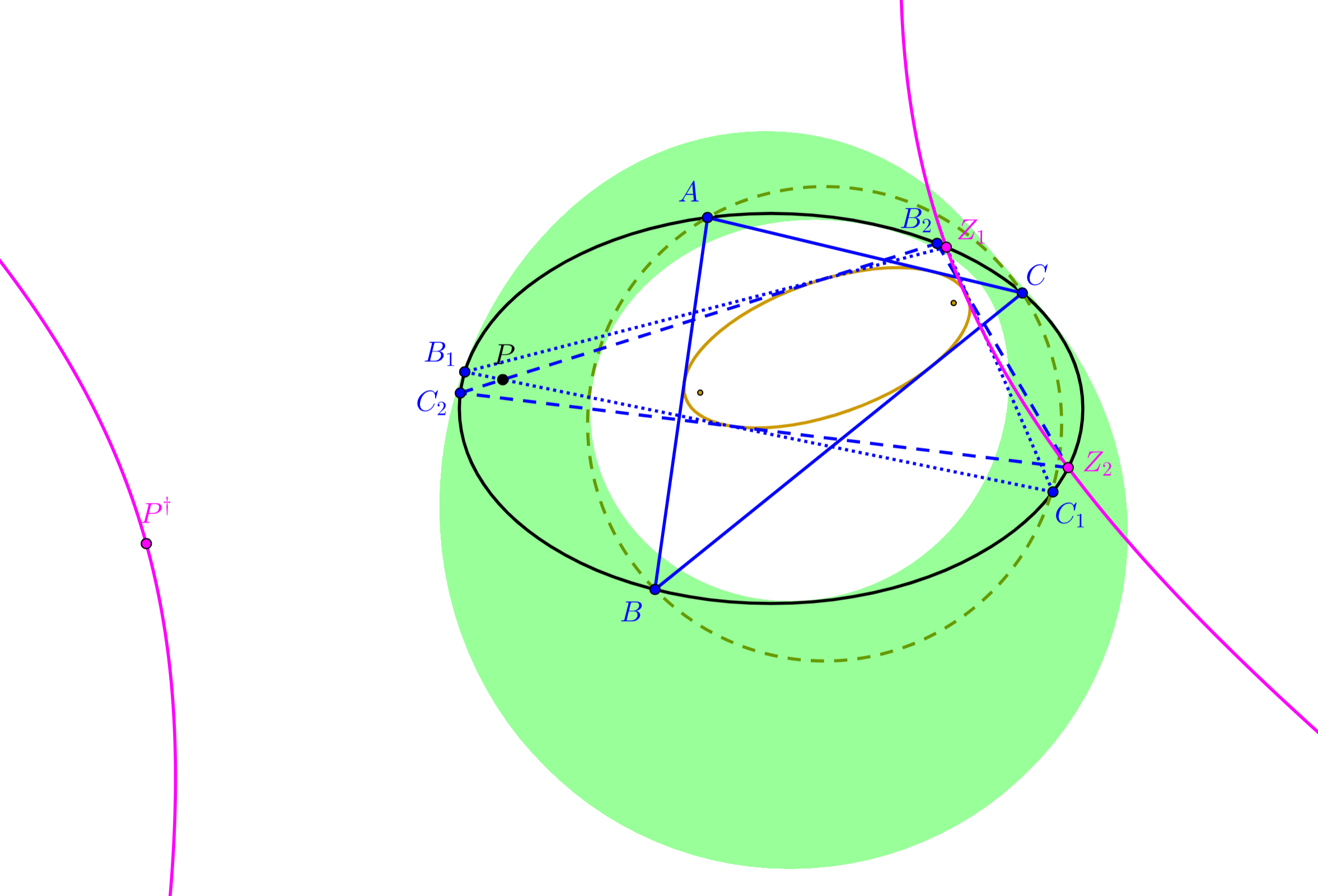}
\zcaption{$P$ is interior to $\Rm$ and closer to $\E$. The two branches of the hyperbolic $\gL$ are shown. With $P$ close to $\E$, the two triangles $B_1 C_1 Z_1$ and $B_2 C_2 Z_2$ approach each other, and $Z_1 Z_2$ tends toward the base of those two triangles.}
\label{fig:two-branches}
\end{figure}

\twofigs{pics/pics_400_p_on_E}{pics/pics_410_p_on_E_envelope}{\tb{left}: $P$ is on $\E$, $\gL$ becomes a line minus the point $Q$ where said line meets with $Z_1 Z_2$. \tb{right}: The envelope of the line-like $\gL$ (not shown) over all $P$ on $\E$ is a conic concentric with $\E_c$ and axis-aligned with $\E$.}{fig:on-E}

\subsection{Degenerate locus}

Referring to \cref{fig:on-E} (left):

\begin{proposition}
If $P$ is on $\E$, the $\gL$ conic degenerates to a line minus a point $Q$ at the intersection of the base $Z_1 Z_2$ of a Poncelet triangle with said line.
\label{prop:line}
\end{proposition}

\begin{proof}
This is verified by evaluating \cref{eqn:conic-eqn} with $P(t)$ on $\E$, with $t$ as in \cref{eqn:rational} and simplifying with a CAS, yielding a degree-1 implicit. The rather long expressions for both $gL$ and $Q$ appear in \cref{app:line-isog}.
\end{proof}


Referring to \cref{fig:on-E} (right):

\begin{proposition}
Over all $P$ on $\E$, the envelope $\gL_{env}$ of the line-like $\gL$ is a conic concentric with $\E_c$ and axis-aligned with $\E$.
\label{prop:line-env}
\end{proposition}

\begin{proof}
Eliminate $t$ from a system comprising (i) the line-locus equation obtained above and (ii) its derivative with respect to $t$. This yields a conic whose center and axes are as claimed. The rather long expression for $\gL_{env}$ is found in \cref{app:line-isog}.
\end{proof}




\section{Two envelopes with a circular caustic}
\label{sec:circ-env}
Let $\To$ denote the Poncelet family inscribed in $\E$ (semiaxes $a,b$) and circumscribing a circular caustic $\K=(C,r)$ with $C=[x_c,y_c]$. The pair $\E,\K$ will admit a family of Poncelet triangles if $r$ is given by \cite[Prop.2]{garcia2024-incircle}:
\begin{equation}
\label{eqn:circ-r}
r=\frac{b\sqrt{a^4-c^2 x_c^2}- a\sqrt{b^4+c^2 y_c^2}}{c^2},
\end{equation}
with $c^2=a^2-b^2$. Let $\L_3$ denote the locus of the circumcenter $X_3$ over $\To$, which must be a conic since it is a fixed (and trivial) linear combination of $X_2$ and $X_3$  \cite[Thm.2]{helman2021-power-loci}.

In \cite[Prop.7]{garcia2024-incircle} we prove that the foci $F_3$ and $F_3'$ of $\L_3$ are collinear with $C$, with each lying on an axis of $\E$ at:  
\begin{equation}
\label{eqn:x3-foci}
F_3=\left[x_c (1-(b/a)^2), 0\right],{\quad}F_3'=\left[0,y_c(1-(a/b)^2)\right]\ldotp
\end{equation}

Expressions for the semi-axis lengths $a_3,b_3$ of $\L_3$ are also derived in \cite[Prop.7]{garcia2024-incircle}:
\begin{equation}
\label{eqn:a3b3}
a_3 = \delta_3(a/b)-\delta_3'(b/a),\quad b_3=\delta_3'-\delta_3,
\end{equation}
with $\delta_3=\frac{\sqrt{b^4+c^2 y_c^2}}{2b}$ and $\delta_3'=\frac{\sqrt{a^4-c^2 x_c^2}}{2a}$.

\subsection{Circumcircle}

Let $\delta=\sqrt{(b^4 +c^2 y_c^2) (a^4 - c^2 x_c^2)}$. Referring to \cref{fig:circum-env}, the following is a specialization of \cref{prop:env}:

\begin{proposition}
Over $\To$, the envelope of the circumcircle is the union of two nested circles, each tangent to $\E$ at two points and centered at the foci $F_3$ and $F_3'$ of $\L_3$.
\label{prop:env-circumcircle}
\end{proposition}

\begin{proof}
Let $C=[x_c,y_c]$ be the center of caustic and $r$ given by \cref{eqn:circ-r}; the incircle $\K$ is then given by $(x-x_c)^2+(y-y_c)^2= r^2$. 
Let $\lambda=\cos{u}+i \sin{u}$ be of the symmetric parametrization of Poncelet triangles as in \cite[Lem.1]{helman2021-theory}. Obtain the circumcircle $\K'$ as a function of $u$:
\begin{align*}
\K'(x,y)=&  x^{2}+y^{2} +\frac{c^{2} x_c \cos{u}}{a}+\frac{ c^{2} y_c\sin{u}}{b}-\frac{\delta}{b a}\\
&-\left(\frac{ \left(a^{3} b -\delta \right)\cos{u}}{b \,a^{2}}+\frac{c^{2} y_c x_c\sin{u} }{b \,a^{2}}+\frac{c^{2} x_c}{a^{2}}\right) x\\
&+\left(-\frac{ c^{2} x_c y_c \cos{u} }{a \,b^{2}}+\frac{c^{2} y_c}{b^{2}}+\frac{\left(a \,b^{3}-\delta \right)\sin{u} }{a \,b^{2}}\right) y = 0 \ldotp\\
\end{align*}
The envelope of this family of circles is obtained by setting:
\[\K'(x,y)=0,\,(\partial{\K'}/\partial u)(x,y)=0 \ldotp\]
With the above equations this yields the union of two circles $\K_1=(O_1,r_1)$ and $\K_2=(O_2,r_2)$ where:
\begin{align*}
O_1 =& [0,-y_cc^2/b^2]=[0,y_c(1-(a/b)^2)]=F_3',{\quad}r_1=(a/b^2)\sqrt{b^4+c^2 y_c^2}~,\\
O_2 = & [x_c c^2/a^2,0]=[x_c(1-(b/a)^2),0] = F_3,{\quad}r_2=(b/a^2)\sqrt{a^4 - c^2 x_c^2}~\ldotp
\end{align*}
$\K_1$ touches $\E$ at 
$\left[\pm (a/b)\sqrt{b^2 - y_c^2},y_c\right]$; $\K_2$ touches $\E$ at $\left[x_c, \pm(b/a)\sqrt{a^2 - x_c^2}\right]$.
\end{proof}

\begin{figure}
\centering
\includegraphics[width=\linewidth]{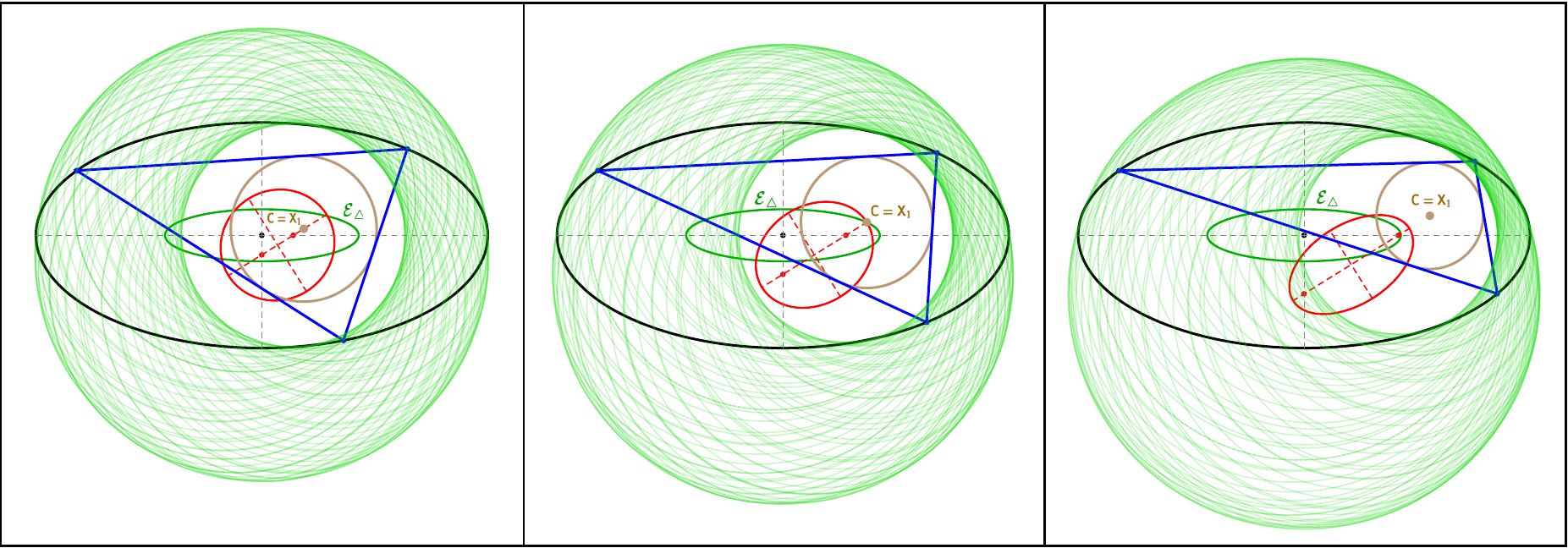}
\zcaption{Over $\To$ (circular caustic, brown), the region swept by the circumcircle (green) is bounded by two nested circles, each twice tangent to $\E$, and centered on the foci of $\L_3$ (red ellipse), notice its foci lie on the axis of $\E$ and its major axis contains $C$. Left, middle, right show three positions for $C=X_1$: interior, on the boundary, and exterior to $\Et$ (dark green), respectively. Video: \hrefs{https://www.youtube.com/watch?v=Y-lagwbuNO8}}
\label{fig:circum-env}
\end{figure}

Given two fixed nested (resp. unnested) circles, the locus of the center of circles simultaneously tangent to them is a conic with foci on their centers and with major axis length equal to the difference (resp. sum) of their radii \cite[thm.14, p.57]{stanley1969-excursions}. In our case, the circumcircles of a Poncelet family circumscribing a circle are tangent to the two circular boundaries of their envelope. It follows from \cref{eqn:a3b3}:

\begin{corollary}
$|r_1-r_2| = 2\,a_3$.
\end{corollary}

For the general case, experimental evidence suggests:

\begin{conjecture}
Over $\T$, either component of the envelope of the circumcircle is a conic if and only if $\E_c$ is a circle, in which case both components are circles (\cref{prop:env-circumcircle}).
\end{conjecture}

Note that if the outer conic is a circle, said envelope is not defined.

\subsection{Incircle-circumcircle radical axis}

Let $\Gamma$ denote the radical axis of incircle and circumcircle of a triangle, i.e., the line whose points have the same power with respect to both circles \cite[Radical line]{mw}. $X_1 X_3$ is perpendicular to ${\Gamma}$ since it is the line that passes through both circle centers. Incidentally, the intersection of $X_1 X_3$ with ${\Gamma}$ is named $X_{3660}$ on \cite{etc}. 

In \cite{garcia2024-incircle} we derived the locus $\L_3$ for the circumcenter $X_3$ over $\To$. 

Referring to \cref{fig:env-ell-hyp}:

\begin{proposition}
Over $\To$, the envelope of ${\Gamma}$ is a conic whose major axis coincides with that of $\L_3$, i.e., along $F_3 F_3'$.
It will be an ellipse (resp. hyperbola) if $C$ is interior (resp. exterior) to $\Et$.
\end{proposition}

\begin{proof}
Let $C=[x_c,y_c]$ be the center of caustic and $r$ the radius of the caustic given by \cref{eqn:circ-r}; the incircle $\K$ is then given by $(x-x_c)^2+(y-y_c)^2= r^2$. Let $\lambda=\cos{u}+i \sin{u}$ in the symmetric parametrization of \cref{sec:symmetric}, obtain the circumcircle $\K'$:
{\small
\begin{align*}
&\K'(x,y)=  x^{2}+y^{2}-\left(\frac{ \left(a^{3} b -\delta \right) }{b \,a^{2}}\cos u+\frac{c^{2} y_c x_c}{b \,a^{2}}\sin{u} +\frac{c^{2} x_c}{a^{2}}\right) x\\
+&\left(-\frac{ c^{2} x_c y_c  }{a \,b^{2}}\cos{u}+\frac{\left(a \,b^{3}-\delta \right)}{a \,b^{2}} \sin{u} +\frac{c^{2} y_c}{b^{2}}\right) y +\frac{c^{2} x_c }{a}\cos{u}+\frac{ c^{2} y_c}{b}\sin{u}-\frac{\delta}{b a}
=0 \ldotp
\end{align*}}

The radical axis of $(\K,\K')$, is given by $l(u)x+m(u)y+n(u)=0$, with:
{\small \begin{align*}
l(u)=&c^4\left(-bx_cy_c \,c^{2} \sin{u}+b \left(\delta -a^{3} b \right)\cos{u}+b^{2}\left(a^{2}+b^{2}\right)x_c \right),\\
m(u)=&c^4\left(a \left(a \,b^{3}-\delta \right) \sin{u}-ax_cy_c \,c^{2}\cos{u}+a^{2}\left(a^{2}+b^{2}\right) y_c \right),\\
n(u)=&a^{2} by_c \,c^{6} \sin{u}+a \,b^{2}x_c \,c^{6} \cos{u}+a^{2} b^{2} c^{2} \left(-a^{2}x_c^{2}+b^{2}y_c^{2}\right)\\&+a^{4} b^{4} \left(a^{2}+b^{2}\right)-a b \left(a^{4}+b^{4}\right)\delta \ldotp \end{align*}
}

The envelope of the family of lines above is given by:
\[ E(u)=\left[
\frac{n'(u)m  \left(u \right) - m'(u) n  \left(u \right)}{m'(u) l  \left(u \right)-l'(u) m  \left(u \right)}, 
-\frac{n'(u) l  \left(u \right)-l'(u) n  \left(u \right)}{m'(u) l  \left(u \right)-l'(u) m  \left(u \right)}
\right] \cdot \]

It follows from the calculations that the envelope is parametrized by:
\[E(u)=\left[ \frac{ k_1\sin u+k_2\cos u+k_3}{d_1\sin u+d_2\cos u+d_3}, \frac{ m_1\sin u+m_2\cos u+m_3}{d_1\sin u+d_2\cos u+d_3}\right] \cdot\]

Here all constants are long expressions involving the variables $(a,b,c,x_c,y_c).$
Eliminating the variables we obtain that the envelope is given implicitly by a quadratic equation and so it is a conic. Algebraic manipulation yields its major axis along $F_3 F_3'$.
\end{proof}

Let $\Et$ denote the locus of centroids of a 1d family of equilateral triangles inscribed in an ellipse $\E=(a,b)$. This locus turns out to be an ellipse, concentric and axis-aligned with $\E$, and with semi-axes $a_{\triangle},b_{\triangle}$ given by \cite{stanev2019-equi} (see animations in \cite{hui2018-equi}, where these are called $a_1,b_1$):
\[ a_{\triangle}=\frac{a\,c^2}{a^2 + 3b^2}\,\rc\quad b_{\triangle}= \frac{b\,c^2}{3a^2 +  b^2}\cdot\]

By continuity, if $C$ is on $\Et$, $\To$ will contain an equilateral, and is henceforth called $\Tt$. In \cite[Prop.9]{garcia2024-incircle}, it is shown that while over $\To$, the locus of $X_{36}$ (the circumcircle-inverse of the incenter $X_1$) sweeps a circle, over $\Tt$ it degenerates to the line:
\[ b^{2} x_c x + a^{2} y_c y = \frac{b^{2} \left(a^{4}+2 a^{2} b^{2}+5 b^{4}\right) x_c^{2}+ a^{2}\left(5 a^{4}+2 a^{2} b^{2}+b^{4}\right) y_c^2}{c^4}\rd\]

Referring to \cref{fig:env-par}, over $\Tt$:

\begin{corollary}
The envelope of ${\Gamma}$ is a parabola whose axis of symmetry is the major axis of $\L_3$, i.e., its directrix is parallel to $\L_{36}$ (a line).
\end{corollary}

\begin{proof}
When the family is at the equilateral position, $X_1$ and $X_3$ coincide, and ${\Gamma}$ is the line at infinity.
\end{proof}

\begin{figure}
\centering
\includegraphics[width=\linewidth]{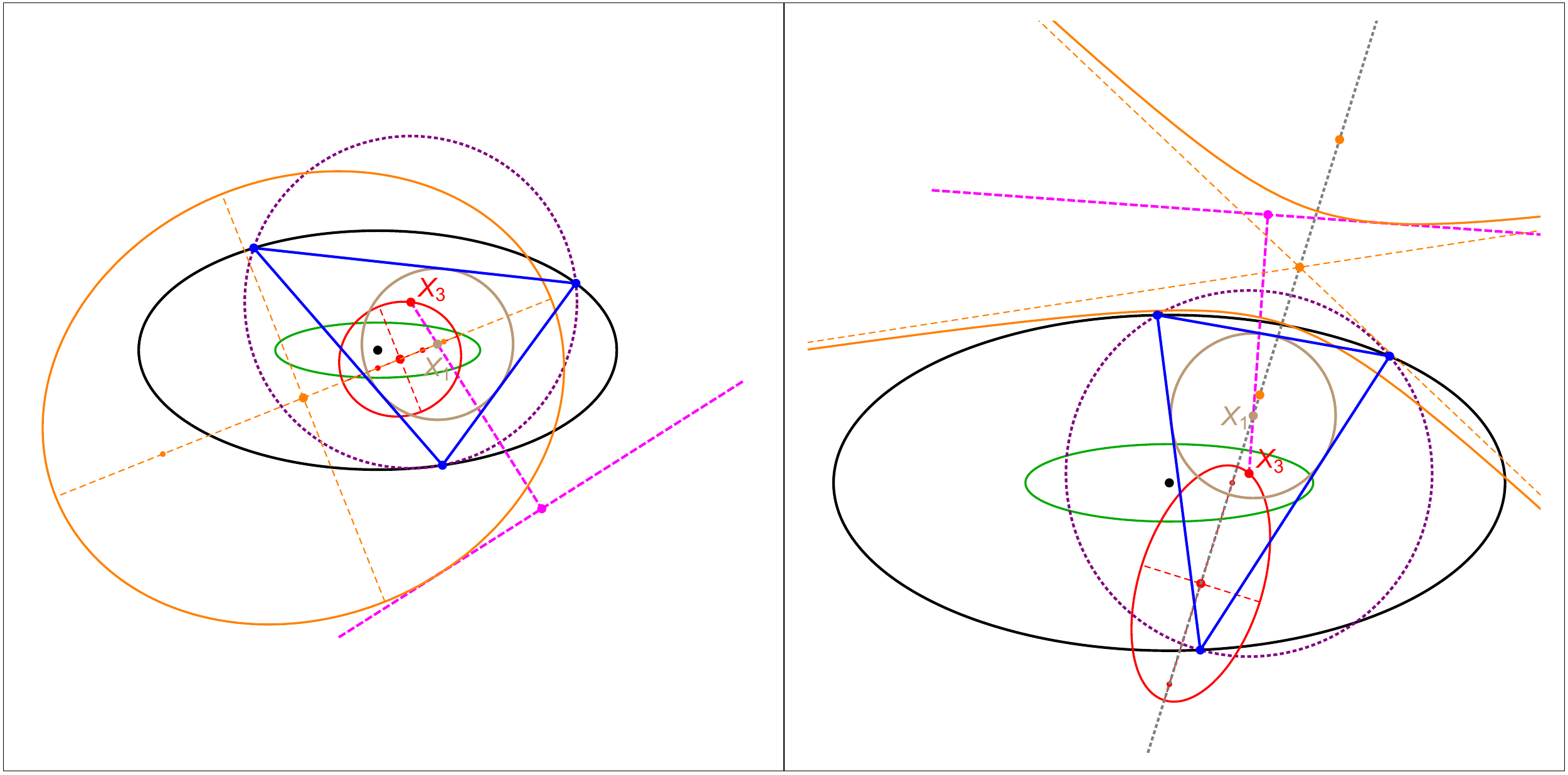}
\zcaption{Left (resp. right): envelope (orange) of ${\Gamma}$ (dashed magenta) is an ellipse (resp. hyperbola), when $C=X_1$ is interior (resp. exterior) to $\Et$ (green ellipse). The major axis of $\L_3$ (red) coincides with that of the envelopes.} 
\label{fig:env-ell-hyp}
\end{figure}

\begin{figure}
\centering
\includegraphics[width=.8\linewidth]{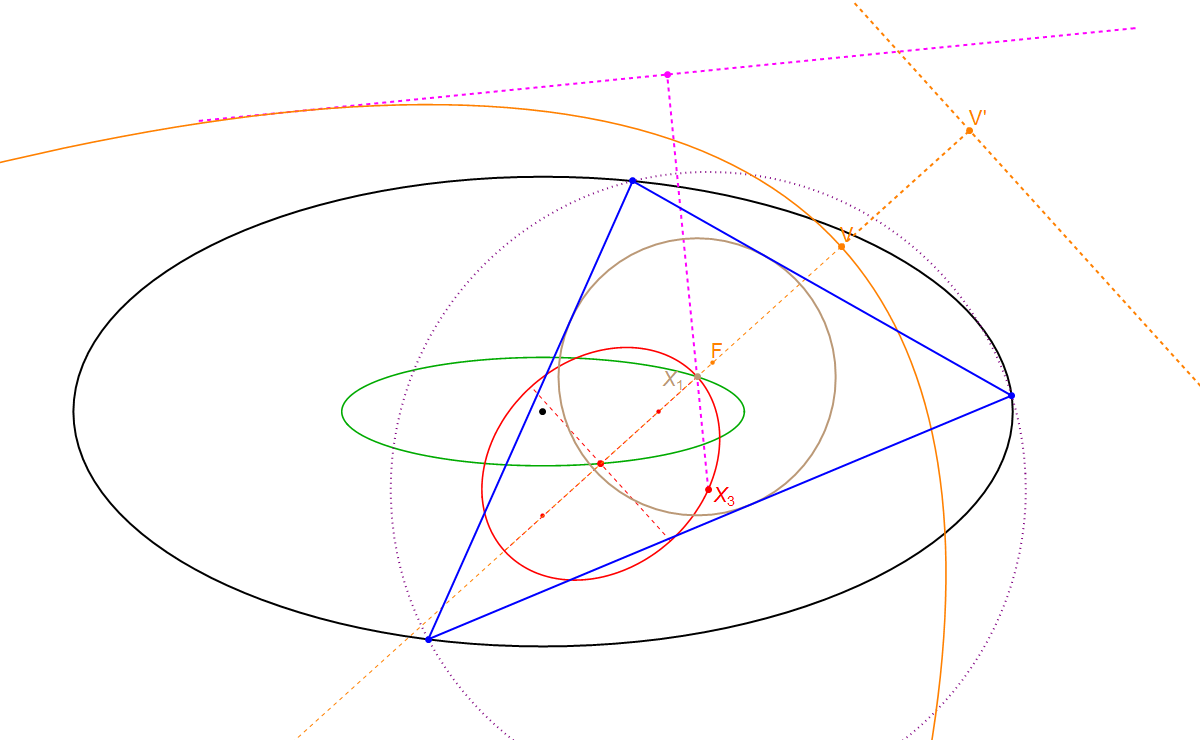}
\zcaption{When $C=X_1$ is on the boundary of $\Et$ (green ellipse), the envelope of ${\Gamma}$ (dashed magenta) is parabolic (orange, focus $F$, vertex $V$); the axis of symmetry (dashed orange) is the major axis of $\L_3$ (red). Also shown is the directrix of the parabola (dashed orange), with $V'$ its intersection with the axis of symmetry.}
\label{fig:env-par}
\end{figure}

Based on experimental evidence:

\begin{conjecture}
Over $\To$, the envelope of ${\Gamma}$ is a conic if and only if the Poncelet caustic is a circle.
\end{conjecture}

The Feuerbach point $X_{11}$ is where incircle and nine-point circle touch. Therefore, the radical axis of said circles is tangent to the incircle at $X_{11}$. Let $L_{101}$ denote said axis as in the table at the end of \cite[Radical line]{mw}. Therefore:

\begin{corollary}
Over $\To$, the envelope of $L_{101}$ is the incircle itself.
\end{corollary}

In \cite[Prop.26]{garcia2024-incircle} we show that over $\Tt$, $X_{11}$ is stationary. Therefore, in such a case said envelope is undefined.

The \ti{orthic axis}, called $L_3$ in (resp. \ti{anti-orthic axis}, $L_1$) in \cite[Radical line]{mw}, is the radical axis of the circumcircle and the nine-point (resp. Bevan) circle. The latter is centered on $X_5$ (resp. $X_{40}$) and has radius $R/2$ (resp. $2R$).

Experimentation shows:

\begin{observation}
Over $\To$, the envelope of both $L_3$ and $L_1$ are conics: an (ellipse, parabola, hyperbola) if $C$ is (interior, on the boundary, exterior) to $\Et$ respectively. Nevertheless, the envelope axes and those of $L_3$ are not axis-aligned.  
\end{observation}



\section*{Acknowledgements}
\noindent We would like to thank 
Arseniy Akopyan, Dominique Laurain, Martin Baptiste, and Liliana Gheorghe for their kind assistance during this project. The first author is a fellow of CNPq.

\appendix

\section{Table of Symbols}
\label{app:symbols}
\begin{table}[H]
\small
\begin{tabular}{|c|p{8cm}|}
\hline
symbol & meaning \\
\hline
$\mathbb{T},\mathbb{D}$ & unit circle (resp. disk) in complex plane \\ 
$f,g$ & foci of caustic of affine image of $\T$ for which $\E$ is a circle \\
$\lambda$ & parameter of symmetric parametrization \\
\hline
$\E, \E_c$ & Poncelet conics (incidence and tangency)  \\
$O, O_c$ & their centers \\
$a,b,c$ & semi-axis' and half-focal lengths of $\E$ \\
$\K=(C,r)$ & circular caustic centered at $C$ and of radius $r$ \\
$\Et$ & elliptic locus of centroids of equilaterals inscribed in $\E$ \\
\hline
$\T$ & Poncelet triangle family between nested $\E,\E_c$ \\
$\To$ & Poncelet triangle family between nested $\E,\K$ \\
$\Tt$ & the family $\To$ containing an equilateral, i.e., $C$ on $\Et$ \\
\hline
$T, l_i, \theta_i$ & a Poncelet triangle, sidelengths, and internal angles \\
$r, R$ & inradius and circumradius of a $T$ \\
$\Rm, \partial\Rm$ & region swept by the circumcircle over $\T$ and its boundary \\
\hline
$X_k,\L_k$ & a triangle center and its locus over some specified family \\
$L_1,L_3$ & anti-orthic and orthic axes of a triangle \\
$L_{101}$ & radical axis of incircle and Euler circle \\
\hline
\end{tabular}
\zcaption{Symbols used in the article.}
\label{tab:symbols}
\end{table}

\section{Triangle centers}
\label{app:centers}
Triangle centers mentioned in the text, see \cite{etc} for details.

\begin{table}[H]
\small
\begin{tabular}{|@{\hspace{2pt}}c@{\hspace{2pt}}|@{\hspace{2pt}}c@{\hspace{2pt}}|@{\hspace{2pt}}l@{\hspace{2pt}}|@{\hspace{2pt}}l@{\hspace{2pt}}|p{8cm}|}
\hline
center & name & first barycentric & construction \\
\hline
$X_1$ & incenter & $l_1::$ &  meet of angle bisectors \\
$X_2$ & barycenter/centroid & $1:: $ & meet of medians  \\
$X_3$ & circumcenter & $l_1^2(l_2^2 + l_3^2 - l_1^2)::$ & meet of perpendicular bisectors \\
$X_4$ & orthocenter & $(l_2^2+l_3^2-l_1^2)^{-1}::$ &  meet of altitudes \\
$X_5$ & Euler center & $l_1^2(l_2^2 +l_3^2)-(l_2^2-l_3^2)^2::$ & center of the Euler circle \\
$X_{11}$ & Feuerbach's point & $(l_2+l_3-l_1)(l_2-l_3)^2::$ & incircle and Euler circle touchpoint \\
$X_{36}$ & circumcircle-inverse of $X_1$  & $l_1^2(l_2^2+l_3^2-l_1^2- l_2 l_3)::$ & isogonal conjugate of $X_{80}$  \\
$X_{40}$ & Bevan point & \makecell[ct]{$l_2/(l_3+l_1-l_2)+$\\$l_3/(l_1+l_2-l_3)-$\\$l_1/(l_2+l_3-l_1)::$} & circumcenter of excentral triangle \\ 
\hline
\end{tabular}
\zcaption{Kimberling codes for various triangle centers mentioned here, along with their names, barycentric coordinates (only the first shown, the other two can be obtaine by cyclical replacement), and construction notes \cite{etc}. The isogonal conjugate of a point with barycentrics $[z_1:z_2:z_3]$ is $[l_1^2/z_1 : l_2^2/z_2 : l_3^2/z_3]$, where the $l_i$ are the sidelengths.}
\label{tab:centers}
\end{table}

\section{Locus constants}
\label{app:factors}
The following are expressions for the constants $s_i,t_i$, $i=0,1,2$ in \cref{eqn:isog-ratio}.
{\allowdisplaybreaks
{\small
\begin{align}
s_2=&-(a^2-b^2) (\ol{f} \ol{g} (a-b)-a-b) \ol{P} P - (2 a+2 b) a b (\ol{f}+\ol{g}) (\ol{f} \ol{g} (a-b)-a-b) \ol{P} \nonumber\\
 -& (2 a+2 b) a b (\ol{f}+\ol{g}) (a-b) P+(a+b)^2 (\ol{f} \ol{g} (a-b)-a-b) \ol{P}^2\nonumber\\
 +& (a+b) a b ((\ol{g}^2+1) (\ol{f}^2+1) a^2-2 (\ol{f}+\ol{g})^2 a b-(1-\ol{g}^2 )(1-\ol{f}^2) b^2),\nonumber\\
s_1=&-2 a b (a^2-b^2) f g P-2 a b (a+b) (\ol{f} \ol{g} (a-b)-a-b) f g \ol{P} \nonumber\\
 +& 2 a b (a+b) (a-b)^2 (\ol{f}+\ol{g}) f g-(a+b) (a-b)^2 f P \ol{P}+(a-b) (a+b)^2 f \ol{P}^2\nonumber\\
-& (2 a+2 b) a b (\ol{f}+\ol{g}) (a-b) \ol{P} f+2 a b (a+b) (\ol{f} \ol{g} (a-b) (a+b)-a^2-b^2) f \nonumber\\
 -& (a+b) (a-b)^2 g P \ol{P}+(a-b) (a+b)^2 g \ol{P}^2-(2 a+2 b) a b (\ol{f}+\ol{g}) (a-b) \ol{P} g \nonumber\\
 +& 2 a b (a+b) (\ol{f} \ol{g} (a-b) (a+b)-a^2-b^2) g+(a-b) (a+b)^2 (\ol{f}+\ol{g}) \ol{P} P \nonumber\\
 -& 2 a b (\ol{f} \ol{g} (a-b) (a+b)-2 a^2-2 b^2) P-(a+b) (a-b)^2 (\ol{f}+\ol{g}) \ol{P}^2 \nonumber\\
 +& 2 a b (a-b) (\ol{f} \ol{g} (a-b)-a-b) \ol{P}-2 a b (a-b) (a^2+b^2) (\ol{f}+\ol{g}),\nonumber\\
s_0=&a b (a+b) (a-b)^2 f^2 g^2-2 a b (a-b) (a+b) f^2 g \ol{P}  + (a-b)^3 \ol{P}^2+a b (a+b) (a-b)^2 \nonumber\\
 +& a b (a-b) (a+b)^2 f^2-2 a b (a-b) (a+b) f g^2 \ol{P}+(a-b) (a+b)^2 f g P \ol{P} \nonumber\\
 -& (a+b) (a-b)^2 f g \ol{P}^2+4 a^2 b^2 (a-b) f g-2 a b (a^2-b^2) f P+2 b a (a-b)^2 f \ol{P} \nonumber\\
+&a b (a-b) (a+b)^2 g^2-2 a b (a^2-b^2) g P+2 a b (a-b)^2 \ol{P} g-(a+b) (a-b)^2 \ol{P} P, \nonumber\\
t_2=& (a^2-b^2) ( 2( \ol{f}+\ol{g}) a b-(P-\ol{P}) (\ol{f} \ol{g}-1) a 
-(P+\ol{P}) (\ol{f} \ol{g}+1) b), \nonumber\\
t_1=& (2 f g+2 \ol{f} \ol{g}-4) a^3 b-(P-\ol{P}) (f+g-\ol{f}-\ol{g}) a^3 \nonumber\\
 -& (P+\ol{P}) (f+g+\ol{f}+\ol{g}) a^2 b+(-2 f g-2 \ol{f} \ol{g}-4) a b^3 \nonumber\\
 +& (P-\ol{P}) (f+g-\ol{f}-\ol{g}) b^2 a+8 a b P \ol{P}+(P+\ol{P}) (f+g+\ol{f}+\ol{g}) b^3, \nonumber\\
t_0=& (2 f+2 g) a^3 b+(P-\ol{P}) (f g-1) a^3-(P+\ol{P}) (f g+1) a^2 b \nonumber\\
 +& (-2 f-2 g) a b^3-(P-\ol{P}) (f g-1) b^2 a+(P+\ol{P}) (f g+1) b^3 \ldotp \nonumber
\end{align}
}
} 

\section{Circumcircle envelope}
\label{app:envelope-circum}
The envelope of the region $\partial\Rm$ swept by the circumcircle is given by the zeros over $x,y$ of:
 {\allowdisplaybreaks
{\small \begin{align}
& -b^6 \left[ (-1 + f_y^2)\,g_x^2 + (-1 + f_y g_y)^2 + f_x^2 (-1 + g_x^2 + g_y^2) \right] x^2 \nonumber\\
& + 2 a^5 b x \left[ 
b \left( g_x (2 + f_y(f_y + g_y)) + f_x (2 + g_y(f_y + g_y)) \right)
- \left( g_x (3 f_y + g_y) + f_x (f_y + 3 g_y) \right) y 
\right] \nonumber\\
& + a^6 \Big[ 
b^2 \big( 
1 - f_y^2 - g_x^2 + f_x^2(-1 + g_x^2) + (-1 + f_y^2) g_y^2 
- 2 f_x g_x (2 + f_y g_y) 
\big) \nonumber\\
& \qquad + 2b (f_x + g_x)(f_y g_x + f_x g_y) y 
- \big( (-1 + f_x g_x)^2 + (-1 + f_x^2) g_y^2 + f_y^2(-1 + g_x^2 + g_y^2) \big) y^2 
\Big] \nonumber\\
& - 4 a^3 b^2 x \Big[ 
b^2 (f_x + g_x + f_y g_x(f_y + g_y) + f_x g_y(f_y + g_y)) 
- b (g_x(3 f_y + g_y) + f_x(f_y + 3 g_y)) y 
+ (f_x + g_x)(x^2 + y^2) 
\Big] \nonumber\\
& + 2 a b^4 x \Big[ 
b^2(f_y + g_y)(f_y g_x + f_x g_y) 
- b (g_x(3 f_y + g_y) + f_x(f_y + 3 g_y)) y 
+ 2(f_x + g_x)(x^2 + y^2) 
\Big] \nonumber\\
& + a^4 b \Big[ 
2 b^3 \big( 
1 + g_x^2 - f_x^2(-1 + g_x^2) + 2 f_y g_y + g_y^2 
+ 2 f_x (g_x + f_y g_x g_y) - f_y^2(-1 + g_y^2) 
\big) \nonumber\\
& \qquad - b \big( 
5 - 4 f_x g_x + (-1 + f_y^2) g_x^2 + f_y g_y(2 + f_y g_y) 
+ f_x^2(-1 + g_x^2 + g_y^2) 
\big) x^2 - 4 b^2 (f_y + f_y g_x(f_x + g_x) + g_y + f_x(f_x + g_x) g_y) y \nonumber\\
& \qquad + 2b \big( 
-1 - 2 f_y g_y - g_y^2 + f_y^2(-1 + g_x^2 + g_y^2) + f_x^2(g_x^2 + g_y^2) 
\big) y^2 
+ 4(f_y + g_y)y(x^2 + y^2) 
\Big] \nonumber\\
& + a^2 b^2 \Big[ 
b^4 \big( 
1 - g_x^2 + f_x^2(-1 + g_x^2) - 2 f_x f_y g_x g_y 
+ (-1 + f_y^2) g_y^2 - f_y(f_y + 4 g_y) 
\big) \nonumber\\
& \qquad + 2 b^3 \big( 
f_y(2 + g_x(f_x + g_x)) + (2 + f_x(f_x + g_x)) g_y 
\big) y - 4 b(f_y + g_y) y(x^2 + y^2) + 4(x^2 + y^2)^2 \nonumber\\
& \qquad + b^2 \Big( 
2 \big( 
-1 - 2 f_x g_x + (-1 + f_y^2) g_x^2 + f_y^2 g_y^2 
+ f_x^2(-1 + g_x^2 + g_y^2) 
\big) x^2 \nonumber\\
& \qquad\qquad\quad 
- \big( 
5 + f_x g_x(2 + f_x g_x) - 4 f_y g_y 
+ (-1 + f_x^2) g_y^2 + f_y^2(-1 + g_x^2 + g_y^2) 
\big) y^2 
\Big)
\Big]\ldotp\nonumber
\end{align}
}}




\section{Line isogonal locus and its envelope}
\label{app:line-isog}
Let a $P(t)$ on $\E$ be located via the rational parametrization of \cref{eqn:rational}. The locus $\gL$ degenerates to the line given by the zeros of:

{\allowdisplaybreaks
{\small
\begin{align}
&\gL(a,b,f_x,f_y,g_x,g_y,t,x,y)=\nonumber\\
& a^5 b \bigg( (t^2+1) f_x^2 (g_x^2+g_y^2-1)+(t^2+1) f_y^2 (g_x^2+g_y^2-1)\nonumber\\
& \qquad -2 f_x \big(3 (t^2+1) g_x+t^2-1\big)+2 f_y (t^2 g_y+g_y+2 t)+t^2 \big(-\left(g_x (g_x+2)+g_y^2-3\right)\big)\nonumber\\
& \qquad +4 t g_y-g_x^2+2 g_x-g_y^2+3 \bigg)\nonumber\\
& \quad - 2 a^3 b^3 \bigg( (t^2+1) f_x^2 (g_x^2+g_y^2-1)+(t^2+1) f_y^2 (g_x^2+g_y^2-1)\nonumber\\
& \qquad -2 f_x \big(t^2 (g_x+1)+g_x-1\big)-2 f_y (t^2 g_y+g_y-2 t)+t^2 \big(-\left(g_x (g_x+2)+g_y^2+5\right)\big)\nonumber\\
& \qquad +4 t g_y-\left(g_x-2\right)g_x-g_y^2-5 \bigg)\nonumber\\
& \quad + 2 a y (a-b) (a+b) \bigg( 2 t \big(a^2 (-f_x g_x+f_y g_y+1)+b^2 (f_x g_x-f_y g_y+3)\big)\nonumber\\
& \qquad + t^2 (a-b) (a+b) (f_y g_x+f_x g_y+f_y+g_y)-(a-b) (a+b) \big(f_y (g_x-1)+(f_x-1) g_y\big) \bigg)\nonumber\\
& \quad - 2 b x (a-b) (a+b) \bigg( a^2 \Big(t^2 \big(-(f_y g_y+g_x+3)\big)+f_x \big(t^2 (g_x-1)+2 t g_y-g_x-1\big)\nonumber\\
& \qquad +2 t f_y g_x+f_y g_y-g_x+3\Big)+b^2 \Big(t^2 \big(f_x (-g_x)+f_y g_y+f_x+g_x-1\big)-2 t (f_y g_x+f_x g_y)\nonumber\\
& \qquad +f_x g_x-f_y g_y+f_x+g_x+1\Big) \bigg)\nonumber\\
& \quad + a b^5 \bigg( (t^2+1) f_x^2 (g_x^2+g_y^2-1)+(t^2+1) f_y^2 (g_x^2+g_y^2-1)\nonumber\\
& \qquad +2 f_x \big(t^2 (g_x-1)+g_x+1\big)+f_y \big(4 t-6 (t^2+1)g_y\big)\nonumber\\
& \qquad +t^2 \big(-\left(g_x (g_x+2)+g_y^2-3\right)\big)+4 t g_y-g_x^2+2 g_x-g_y^2+3 \bigg)\nonumber
\end{align}}}

For every $P(t)$, the line-locus meets the base of the Poncelet triangle with apex on $P$ at a distinct point $Q(t)=[Q_x,Q_y/2]/\Delta$, where:

{\allowdisplaybreaks
{\small
\begin{align*}
\Delta=&(a-b) (a+b) \Big(-2 t^3 (f_x+g_x+2)+(t^2+1) f_y (t^2 (g_x+1)+g_x-1)\nonumber\\
& \qquad +(t^2+1) g_y (t^2 (f_x+1)+f_x-1)-2 t (f_x+g_x-2)\Big),\nonumber\\
Q_x=& a^3 \Big(t (t^2+1) f_y^2 (g_x^2+g_y^2-1)-(t^2-1) f_y (t^2 (g_x+1)+g_x-1)\nonumber\\
& \qquad + t (t^2+1) (f_x^2-1) g_y^2-(t^2-1) g_y (t^2 (f_x+1)+f_x-1)\nonumber\\
& \qquad +t \big(-2 f_x (2 (t^2+1) g_x+t^2-1)+(t^2+1) f_x^2 (g_x^2-1)\nonumber\\
& \qquad -g_x (t^2 (g_x+2)+g_x-2)+t^2+1\big)\Big)\nonumber\\
& \quad - a b^2 \Big(t (t^2+1) f_y^2 (g_x^2+g_y^2-1)\nonumber\\
& \qquad +t \big(t g_y (t^2 (f_x+1)+6)+(t^2+1) (f_x^2-1) g_y^2+(t^2+1) (f_x^2-1) g_x^2\nonumber\\
& \qquad +t^2 (-(f_x (f_x+2)+3))-2 (t^2-1) g_x\big)\nonumber\\
& \qquad +f_y ((t^4-1) g_x-4 (t^3+t) g_y+t^4+6 t^2+1)\nonumber\\
& \qquad -f_x g_y-t ((f_x-2) f_x+3)+g_y\Big),\nonumber\\
Q_y=& b^3 \Big( (t^4-1) f_x^2 (g_x^2+g_y^2-1)+(t^4-1) f_y^2 (g_x^2+g_y^2-1)\nonumber\\
& \quad -4 t f_x (t^2 g_y+g_y-2 t)-4 f_y ((t^4-1) g_y+t^3 (g_x-1)+t g_x+t)\nonumber\\
& \quad + t^4 (-(g_x^2+g_y^2-1))+4 t^3 g_y+8 t^2 g_x-4 t g_y+g_x^2+g_y^2-1 \Big)\nonumber\\
& - a^2 b \Big( (t^4-1) f_x^2 (g_x^2+g_y^2-1)+(t^4-1) f_y^2 (g_x^2+g_y^2-1)\nonumber\\
& \quad +4 t f_y (t^2 (g_x+1)+g_x-1)-4 f_x ((t^4-1) g_x-(t^3+t) g_y+t^4+1)\nonumber\\
& \quad + t^4 (-(g_x (g_x+4)+g_y^2+3))+4 t^3 g_y-4 t g_y+(g_x-4) g_x+g_y^2+3 \Big).\nonumber
\end{align*}}}

Over all $P$ on $\E$, the envelope $\gL_{env}$ of the line locus is a conic centered on $O_c$ given by the zeros of:

{\allowdisplaybreaks
{\small
\begin{align}
&\gL_{env}(a,b,f_x,f_y,g_x,g_y,x,y)=\nonumber\\
& 4 \bigg(8 a b x y (f_x-g_x)(f_y-g_y)(a^2-b^2)^4 + 4 a^2 (a-b)^2 b (a+b)^2 y \Big( \Big( (f_x^2+f_y^2-1)g_y^3 + f_y(f_x^2+f_y^2-1)g_y^2\nonumber\\
& \qquad + \big((g_x^2+1)f_x^2 - 2g_x f_x + (f_y^2-1)(g_x^2-1)\big)g_y + f_y(f_x^2+f_y^2+1)g_x^2 - f_y(f_x^2+f_y^2-1) - 2f_x f_y g_x \Big) a^4\nonumber\\
& \qquad + 2 b^2 \Big( - (f_x^2+f_y^2-1)g_y^3 - f_y(f_x^2+f_y^2-5)g_y^2 - \big(f_x^2+2g_x f_x-5f_y^2+(f_x^2+f_y^2-1)g_x^2-3\big)g_y\nonumber\\
& \qquad + f_y \big(f_x^2-2g_x f_x+f_y^2-(f_x^2+f_y^2+1)g_x^2+3\big) \Big) a^2 + b^4 \Big( (f_x^2+f_y^2-1)g_y^3+f_y(f_x^2+f_y^2-9)g_y^2\nonumber\\
& \qquad + \big(f_x^2+6g_x f_x-9f_y^2+(f_x^2+f_y^2-1)g_x^2+9\big)g_y + f_y(f_x^2+f_y^2+1)g_x^2-f_y(f_x^2+f_y^2-9)+6f_x f_y g_x \Big) \Big)\nonumber\\
& \quad - 4 (a-b)^2 b^2 (a+b)^2 x^2 \Big( \big((g_x^2+g_y^2-1)f_x^2 - 8g_x f_x + (f_y^2-1)g_x^2 + (f_y g_y+3)^2\big) a^4\nonumber\\
& \qquad - 2b^2 \big( (g_x^2+g_y^2-1)f_x^2 - 4g_x f_x + (f_y^2-1)g_x^2 + f_y g_y(f_y g_y+2) - 3 \big) a^2 + b^4 \big( (g_x^2+g_y^2-1)f_x^2+(f_y^2-1)g_x^2\nonumber\\
& \qquad + (f_y g_y-1)^2 \big) \Big) - 4 a^2 (a-b)^2 (a+b)^2 y^2 \Big( \big((g_x^2+g_y^2-1)f_y^2 + (f_x g_x-1)^2 + (f_x^2-1)g_y^2\big) a^4\nonumber\\
& \qquad - 2b^2 \big( (g_x^2+g_y^2-1)f_y^2 - 4g_y f_y + (f_x^2-1)g_y^2 + (f_x g_x-1)(f_x g_x+3) \big) a^2\nonumber\\
& \qquad + b^4 \big( (g_x^2+g_y^2-1)f_y^2-8g_y f_y+(f_x g_x+3)^2+(f_x^2-1)g_y^2 \big) \Big)\nonumber\\
& \quad + 4 a (a-b)^2 b^2 (a+b)^2 x \Big( \Big( (g_x^2+g_y^2-1)f_x^3 + g_x(g_x^2+g_y^2-9)f_x^2 + \big((f_y^2+1)g_y^2+6f_y g_y+(f_y^2-9)(g_x^2-1)\big)f_x\nonumber\\
& \qquad + g_x\big((g_x^2+g_y^2+1)f_y^2+6g_y f_y-g_x^2-g_y^2+9\big) \Big) a^4 - 2b^2 \Big( (g_x^2+g_y^2-1)f_x^3 + g_x(g_x^2+g_y^2-5)f_x^2\nonumber\\
& \qquad + \big((g_x^2+g_y^2-1)f_y^2+2g_y f_y-5g_x^2+g_y^2-3\big)f_x + g_x\big((g_x^2+g_y^2+1)f_y^2+2g_y f_y-g_x^2-g_y^2-3\big) \Big) a^2\nonumber\\
& \qquad + b^4 \Big( (g_x^2+g_y^2-1)f_x^3 + g_x(g_x^2+g_y^2-1)f_x^2 + \big((f_y^2+1)g_y^2-2f_y g_y+(f_y^2-1)(g_x^2-1)\big)f_x\nonumber\\
& \qquad + g_x\big((g_x^2+g_y^2+1)f_y^2-2g_y f_y-g_x^2-g_y^2+1\big) \Big) \Big)\nonumber\\
& \quad + a^2 b^2 \bigg( \Big( (g_x^2+g_y^2-1)^2 f_x^4-12g_x(g_x^2+g_y^2-1)f_x^3 + 2 \big(22g_x^2+f_y^2(g_x^2+g_y^2-1)^2+2f_y g_y(g_x^2+g_y^2-1)\nonumber\\
& \qquad - (g_x^2+(g_y-2)g_y)(g_x^2+g_y(g_y+2))-5\big)f_x^2 - 4g_x\big(3(g_x^2+g_y^2-1)f_y^2+6g_y f_y-3g_x^2-3g_y^2+11\big)f_x\nonumber\\
& \qquad + f_y^4(g_x^2+g_y^2-1)^2+4f_y^3 g_y(g_x^2+g_y^2-1) - 4f_y g_y(g_x^2+g_y^2-1) + (g_x^2+g_y^2-9)(g_x^2+g_y^2-1)\nonumber\\
& \qquad - 2f_y^2 \big(g_x^4-4g_x^2+g_y^4+2(g_x^2-3)g_y^2+5\big) \Big) a^8 - 4 b^2 \Big( (g_x^2+g_y^2-1)^2 f_x^4 - 8g_x(g_x^2+g_y^2-1)f_x^3\nonumber\\
& \qquad + 2 \big(6g_x^2+f_y^2(g_x^2+g_y^2-1)^2-(g_x^2+g_y^2)^2-1\big)f_x^2 + 8g_x \big(-(g_x^2+g_y^2-1)f_y^2+g_y f_y+g_x^2+g_y^2+2\big)f_x\nonumber\\
& \qquad + f_y^4(g_x^2+g_y^2-1)^2 - 24f_y g_y + (g_x^2+g_y^2-5)(g_x^2+g_y^2+3) - 2f_y^2 \big(2g_y^2+(g_x^2+g_y^2)^2+1\big) \Big) a^6\nonumber\\
& \qquad + 2 b^4 \Big( 3(g_x^2+g_y^2-1)^2 f_x^4-12g_x(g_x^2+g_y^2-1)f_x^3 + 2 \big(-6g_x^2-4g_y^2+3f_y^2(g_x^2+g_y^2-1)^2-3(g_x^2+g_y^2)^2\nonumber\\
& \qquad - 6f_y g_y(g_x^2+g_y^2-1)+1\big)f_x^2+4g_x\big(-3(g_x^2+g_y^2-1)f_y^2 + 14g_y f_y+3g_x^2+3g_y^2+1\big)f_x+2g_x^2+2g_y^2\nonumber\\
& \qquad + 3f_y^4(g_x^2+g_y^2-1)^2 - 12f_y^3 g_y(g_x^2+g_y^2-1) + 4f_y(3g_y^3+3g_x^2 g_y+g_y)-2f_y^2\big(3g_x^4+4g_x^2+3g_y^4\nonumber\\
& \qquad + 6(g_x^2+1)g_y^2-1\big) + 59 \Big) a^4 - 4b^6 \Big( (g_x^2+g_y^2-1)^2 f_x^4 + 2 \big(-2g_x^2+f_y^2(g_x^2+g_y^2-1)^2\nonumber\\
& \qquad - (g_x^2+g_y^2)^2 - 4f_y g_y(g_x^2+g_y^2-1)-1\big)f_x^2+8g_x(f_y g_y-3)f_x + f_y^4(g_x^2+g_y^2-1)^2\nonumber\\
& \qquad - 8f_y^3 g_y(g_x^2+g_y^2-1) + 8f_y g_y(g_x^2+g_y^2+2)+(g_x^2+g_y^2-5)(g_x^2+g_y^2+3)\nonumber\\
& \qquad - 2f_y^2\big(-6g_y^2+(g_x^2+g_y^2)^2+1\big) \Big) a^2 + b^8 \Big( (g_x^2+g_y^2-1)^2 f_x^4 + 4g_x(g_x^2+g_y^2-1)f_x^3\nonumber\\
& \qquad + 2\big(6g_x^2+f_y^2(g_x^2+g_y^2-1)^2 - 6f_y g_y(g_x^2+g_y^2-1)-(g_x^2+(g_y-2)g_y)(g_x^2+g_y(g_y+2))-5\big)f_x^2\nonumber\\
& \qquad + 4g_x\big((f_y^2-1)g_y^2-6f_y g_y+(f_y^2-1)(g_x^2-1)\big)f_x + f_y^4(g_x^2+g_y^2-1)^2 - 12f_y^3 g_y(g_x^2+g_y^2-1)\nonumber\\
& \qquad + (g_x^2+g_y^2-9)(g_x^2+g_y^2-1) + 4f_y g_y(3g_x^2+3g_y^2-11) - 2f_y^2\big(g_x^4-4g_x^2+g_y^4+2(g_x^2-11)g_y^2+5\big) \Big) \bigg) \bigg)\nonumber
\end{align}

\bibliographystyle{maa}
\bibliography{refs,refs_00_book,refs_01_pub,refs_03_sub}

\end{document}